\documentclass[11pt]{amsart}
\usepackage{latexsym,amssymb,amsthm,amsmath,amscd}

\theoremstyle{plain}
\newtheorem{theorem}{Theorem}[section]
\newtheorem*{Theorem B}{Theorem B}
\newtheorem*{Theorem A}{Theorem A}

\newtheorem{lemma}{Lemma}[section]

\numberwithin{equation}{section}

\theoremstyle{remark}

 \numberwithin{equation}{section}

\title{Geometry of quasi-sum production functions with constant elasticity of substitution property}

\author{ Bang-Yen Chen}
\address{\it Michigan State University \newline\indent
Department of Mathematics  \newline\indent
619 Red Cedar Road,  \newline\indent East Lansing, Michigan
48824--1027, U.S.A.}
\email{bychen@math.msu.edu}



\begin{document}

\begin{abstract}
 A production function $f$ is called quasi-sum if there are  strict monotone functions $F, h_1,\ldots,h_n$ with $F'>0$ such that $$f(x)= F(h_1
(x_1)+\cdots +h_n (x_n)).$$ The justification for studying quasi-sum production functions is that these functions appear as solutions of the
general bisymmetry equation and they are related to the problem of consistent aggregation.

In this article, first we present the classification of quasi-sum production functions satisfying the constant elasticity of substitution property. 
Then we prove that if a quasi-sum production function  satisfies the constant elasticity of substitution property, then its graph has vanishing Gauss-Kronecker curvature (or its graph is a flat space) if and only if the production function is either a linearly homogeneous generalized ACMS  function or a linearly homogeneous generalized Cobb-Douglas  function.\end{abstract}

\maketitle
\pagestyle{myheadings}
\markboth{\hspace{0,45 cm}\hrulefill\ B.-Y. Chen}{{\it Geometry of quasi-sum production functions} \hrulefill\hspace{0,45 cm} }


\setcounter{equation}{0}
\section{\uppercase{Introduction}}

In economics, a {production function} is a positive function with non-vanishing first derivatives that specifies the output of a firm, an industry, or an entire economy for all combinations of inputs. Almost all economic theories presuppose a production function, either on the firm level or the aggregate level. In this sense, the production function is one of the key concepts of mainstream neoclassical theories.
By assuming that the maximum output technologically possible from a given set of inputs is achieved, economists using a production function in analysis are abstracting from the engineering and managerial problems inherently associated with a particular production process. 

Let $\mathbb R$ denote the set of real numbers. Put 
$$ {\mathbb R}_{+}=\{r\in {\mathbb R}:r>0\}\;\;{\rm and}\;\;  {\mathbb R}^{n}_{+}=\{(x_{1},\ldots,x_{n})\in {\mathbb R}^{n}:x_{1},\ldots,x_{n}>0\}.$$
By a {\it production function} we mean
a twice differentiable function $f: D\subset {\mathbb R}^{n}_{+}\to {\mathbb R}_{+}$ from a domain $D$  of ${\mathbb R}^{n}_{+}$ into ${\mathbb R}^+$ which has non-vanishing first derivatives. 

  A  production function $f$ is called quasi-sum if there exist strict monotone functions $F, h_1,\ldots,h_n$ with $F'>0$ such that $$f({\bf x})= F(h_1
(x_1)+\cdots +h_n (x_n)).$$
The justification for studying quasi-sum production functions is that these functions appear as solutions of the
general bisymmetry equation and they are related to the problem of consistent aggregation (see \cite{AM}). 

The most common quantitative indices of production factor substitutability are forms of the elasticity of substitution. The elasticity of substitution was originally introduced by J.~R. Hicks \cite{H} in case of two inputs for the purpose of analyzing changes in the income shares of labor and capital. 
R.~G. Allen and J.~R. Hicks \cite{AH} suggested a generalization of Hicks' original two variable
elasticity concept  as follows:
For a production function $f$, put
\begin{align} \label{1.1}H_{ij}({\bf x})=\text{\small$\frac{\dfrac{1}{x_{i}f_{x_{i}}}+\dfrac{1}{x_{j}f_{x_{j}}}}{-\dfrac{f_{x_{i}x_{i}}}{f_{x_{i}}^{2}}+\dfrac{2f_{x_{i}x_{j}}}{f_{x_{i}}f_{x_{j}}}-\dfrac{f_{x_{j}x_{j}}}{f_{x_{j}}^{2}}}$} \end{align}
for ${\bf x}\in {\mathbb R}^{n}_{+},\, 1\leq i\ne j\leq n$, where
 the subscripts of $f$ denote partial derivatives.
All partial derivatives are taken at the point ${\bf x}$ and the denominator is assumed
to be different from zero.
  The  function $H_{ij}$ is called the {\it Hicks elasticity of substitution} of the $i$-th
production variable (input) with respect to the $j$-th production variable (input). 

A  production function $f$ is said to satisfy the constant elasticity of substitution property if there exists a nonzero constant $\sigma\in {\mathbb R}$ such that
 \begin{align} \label{1.2}H_{ij}({\bf x})=\sigma\;\; \text{for ${\bf x}\in \mathbb R^{n}_{+}$ and $1\leq i\ne j\leq n$}.\end{align}
 
In this article, first we present the classification result of quasi-sum production functions satisfying the constant elasticity of substitution property. 
Then we prove that if a quasi-sum production function  satisfies the constant elasticity of substitution property, then its graph has vanishing Gauss-Kronecker curvature (or its graph is a flat space) if and only if the production function is either a linearly homogeneous generalized ACMS  function or a linearly homogeneous generalized Cobb-Douglas  function.

\section{\uppercase{Cobb-Douglas  and  ACMS production functions}}

 Cobb and  Douglas  introduced in \cite{CD} a famous two-input production function
 \begin{align} Y = bL^{k}C^{1-k},\end{align}
 where $b$  represents the total factor productivity, $Y$ the total production, $L$ the labor input and $C$ the capital input.
The Cobb-Douglas production function is widely used  in  economics to represent the relationship of an output to inputs.
 Later work in the 1940s prompted them to allow for the exponents on $C$ and $L$ vary, which resulting in estimates that subsequently proved to be very close to improved measure of productivity developed at that time (cf. \cite{D,Fi}).
In its generalized form the Cobb-Douglas production function may be expressed as
\begin{align} f({\bf x}) =\gamma x_{1}^{\alpha_{1}}\cdots x_{n}^{\alpha_{n}},\end{align}
where $\gamma$ is a positive constant and $\alpha_{1},\ldots,\alpha_{n}$ are nonzero constants. 

A production function $f$ is said to be {\it $d$-homogeneous} or {\it homogeneous of degree} $d$, if  
  \begin{align}\label{1}f(tx_{1},\ldots,tx_{n}) = t^{d}f(x_{1},\ldots,x_{n})\end{align}
holds for each $t\in \mathbb R$ for which \eqref{1} is defined.   A homogeneous function of degree one is called  {\it linearly homogeneous}.  

If $d>1$, the homogeneous function exhibits increasing returns to scale, and it exhibits decreasing returns to scale if $d<1$. A linearly homogeneous production function with inputs capital and labor has the properties that the marginal and average physical products of both capital and labor can be expressed as functions of the capital-labor ratio alone. Moreover, in this case if each input is paid at a rate equal to its marginal product, the firm's revenues will be exactly exhausted and there will be no excess economic profit.

Arrow,   Chenery,  Minhas and  Solow introduced in  \cite{ACMS} a two-factor production function given by
\begin{align}\label{1.3}Q=F\cdot (a K^{r}+(1-a) L^{r})^{\frac{1}{r}},\end{align}
where $Q$ is the output, $F$ the factor productivity, $a$ the share parameter, $K,L$ the primary production factors (capital and labor), $r=(s-1)/s$, and $s=1/(1-r)$ is the elasticity of substitution.
A {\it generalized ACMS production function}  is defined  as
\begin{align}\label{1.4} f({\bf x})=\gamma \Big(\sum_{i=1}^{n}a_{i}^{\rho}x_{i}^{\rho}\Big)^{\! \frac{d}{\rho}}, \;\;{\bf x}=( x_{1},\ldots,x_{n})\in D\subset {\mathbb R}^{n}_{+},\end{align}
 with $a_{1},\ldots,a_{n},\gamma,\rho \ne 0$, where $d$ is the degree of homogeneity. 

In this paper, by a {\it homothetic production function} we mean a production function of the form:  
\begin{align} f({\bf x})=F(h(x_1,\ldots,x_n)),\end{align}
where $F$ is a strictly increasing function and  $h(x_1,\ldots,x_n)$ is a homogeneous function of any given degree $d$. A homothetic production function of form
\begin{align}\label{2} f({\bf x})=F\Big(\sum_{i=1}^{n}a_{i}^{\rho}x_{i}^{\rho}\Big),\;\; \text{(resp., $f({\bf x})=F(x_{1}^{\alpha_{1}}\cdots x_{n}^{\alpha_{n}})$}), \end{align}
 is called a {\it homothetic generalized ACMS production function}  (resp., {\it homothetic generalized Cobb-Douglas production function}).

\section{\uppercase {Production functions with constant elasticity of substitution property}}

\begin{theorem}\label{T:1.1} Let $ f({\bf x})= F(h_1 (x_1)+\cdots +h_n (x_n)), n\geq 2,$ be a quasi-sum production function. Then $f$ satisfies the constant elasticity of substitution property if and only if, up to suitable translations,  $f$ is one of the following functions:
 
\begin{enumerate}
\item  a homothetic generalized ACMS production function given by
\begin{align}\label{1.5} f({\bf x})=F\left(c_1 x_{1}^{\frac{\sigma-1}{\sigma}}+\cdots+c_n x_{n}^{\frac{\sigma-1}{\sigma}} \right),\;\;\sigma \ne 1; \end{align} 
\item a homothetic generalized Cobb-Douglas production function given by
\begin{align}\label{11.5} f({\bf x})=F\left( x_{1}^{\alpha_{1}}\cdots x_{n}^{\alpha_{n}} \right); \end{align} 
\item  a two-input production function of the form:
\begin{align}\label{1.6} f({\bf x})=F\! \left(\frac{x_2}{x_1}\right),\end{align}
\end{enumerate}
where $F$ is a strictly increasing function.

\end{theorem}
 \begin{proof}  
  Let 
\begin{align}\label{2.1} f({\bf x})= G(h_1 (x_1)+\cdots +h_n (x_n))\end{align}
  be a quasi-sum production function. Then $G,h_1,\ldots,h_n$ are  strict monotone functions with $G'>0$. Thus we have       
  \begin{align}\label{2.2} G'(u),h_1'(x_1),\ldots,h'_n(x_n)\ne 0\end{align}  
 with  $u=h_1 (x_1)+\cdots +h_n (x_n).$
  From  \eqref{2.1} we find
    \begin{equation} \begin{aligned}\label{2.3} &f_{x_i x_i}=G' h_i^{''}+F'' h_i'{}^2,\;\; 
    \; f_{x_ix_j}=G'' h_i'h_j',\;\; 1\leq i\ne j\leq n.\end{aligned}  \end{equation}
   Assume that $f$ satisfies the constant elasticity of substitution property. Then there exists a nonzero constant $\sigma$ such that
 \begin{equation} \begin{aligned}\label{2.4} &2 f_{x_{i}}f_{x_{j}}f_{x_{i}x_{j}}-f^{2}_{x_{j}} f_{x_{i}x_{i}}-f_{x_{i}}^{2} f_{x_{j}x_{j}}=\frac{(x_{i}f_{x_{i}}+x_{j}f_{x_{j}})f_{x_{i}}f_{x_{j}}}{\sigma x_{i} x_{j}}
\end{aligned}  \end{equation} for $1\leq  i< j\leq n.$
It follows from \eqref{2.3} and \eqref{2.4} that
\begin{equation} \label{2.5} \frac{1}{x_i h_i'(x_i)}+ \frac{\sigma  h_i'' (x_i) }{h_i'(x_i){}^2} + \frac{1}{x_j h_j'(x_j)}+ \frac{\sigma  h_j'' (x_j) }{h_j'(x_j){}^2}  =0,\;\; 1\leq  i< j\leq n.
 \end{equation}

\vskip.04in 

Now, we divide the proof into two separate cases.
\vskip.05in 

{\bf Case} (a): $n\geq 3$. It follows from \eqref{2.5} that
\begin{align}\label{2.6} &  \frac{1}{x_i h_i'(x_i)}+ \frac{\sigma  h_i'' (x_i) }{h_i'(x_i){}^2}  =0,\;\; i=1,\ldots,n.\end{align}
After solving  \eqref{2.6}  we find
\begin{align} \label{2.7}h_i(x_i)=
\begin{cases} a_i+c_i x_i^{\frac{\sigma-1}{\sigma}} &\text{if $\sigma\ne 1$;}
\\ a_{i}\ln (c_{i}x_{i}) &\text{if $\sigma=1$},\end{cases}
\end{align}
for some constants $a_i,c_i$.  Combining \eqref{2.1} and \eqref{2.7} gives cases (1) and (2) of the theorem.

\vskip.05in

{\bf Case} (b): $n=2$. First we consider the case:
\begin{align}\label{2.8} &  \frac{1}{x_i h_i'(x_i)}+ \frac{\sigma  h_i'' (x_i) }{h_i'(x_i){}^2}  =0\end{align}
for $i=1,2$. After solving \eqref{2.8} we obtain \eqref{1.5} with $n=2$ in the same way as case (a).

Next, let us assume that \eqref{2.8} does not hold. Then it follows from \eqref{2.5} that
\begin{equation}\begin{aligned} \label{2.9}& \frac{1}{x_1 h_1'(x_1)}+ \frac{\sigma  h_1'' (x_1) }{h_1'(x_1){}^2} =k,\;\;\frac{1}{x_2 h_2'(x_2)}+ \frac{\sigma  h_2'' (x_2) }{h_2'(x_2){}^2}  =-k
\end{aligned} \end{equation}
for some nonzero constant $k$.
By solving  \eqref{2.9}  we derive that 
\begin{equation}\begin{aligned}  \label{2.10} &h_1(x_1)=a_1-\frac{\sigma}{k}\ln \left(b_1 +k x_1^{\frac{\sigma-1}{\sigma}}\right),
\;\; h_2(x_2)=a_2+\frac{\sigma}{k}\ln \left(b_2 +k x_2^{\frac{\sigma-1}{\sigma}}\right),
\end{aligned} \end{equation}
for some constants $a_1,a_2,b_1,b_2$. Thus, by applying suitable translations on $x_1,x_2$, we find
\begin{equation}\begin{aligned}  \label{2.11} &h_1(x_1)=a_1-\frac{\sigma}{k}\ln \left(k x_1^{\frac{\sigma-1}{\sigma}}\right),
\;\; h_2(x_2)=a_2+\frac{\sigma}{k}\ln \left(k x_2^{\frac{\sigma-1}{\sigma}}\right),
\end{aligned} \end{equation}
Now, by combining \eqref{2.1} and \eqref{2.11} we have
\begin{equation}\begin{aligned}  \label{2.12} &f=G\left(a_0-\frac{\sigma}{k}\ln \left(k x_1^{\frac{\sigma-1}{\sigma}}\right)+\frac{\sigma}{k}\ln \left(k x_2^{\frac{\sigma-1}{\sigma}}\right)\right)
\\& \hskip.1in =G\left(a_0-\frac{\sigma-1}{k}\ln \left(\frac{x_2}{x_1}\right)\right),\;\; a_0=a_1+a_2.
\end{aligned} \end{equation}
 Consequently, we get a two-input  function of the form \eqref{1.6}. This gives case (3).

The converse is easy to verify.\end{proof}

Theorem \ref{T:1.1} is a reformulation of a result of \cite{L}.

\section{\uppercase{A characterization of generalized ACMS and Cobb-Douglas production functions}}

Each $n$-input production function $f({\bf x})$ can be identified with the {\it graph} of $f$  defined as the non-parametric hypersurface of the Euclidean $(n+1)$-space $\mathbb E^{n+1}$ given by (cf. \cite{c,VV,V})
\begin{align}\label{3.1} L({\bf x})=(x_{1},\ldots,x_{n},f(x_{1},\ldots,x_{n})).\end{align}

For  a hypersurface $M$ of $\mathbb E^{n+1}$, the {\it Gauss map} $$\nu: M\to S^{n+1}$$ is a map which carries  $M$  to the unit hypersphere $S^{n}$ of  $\mathbb E^{n+1}$. The Gauss map is a continuous map  such that $\nu(p)$ is a unit normal vector $\xi(p)$ of $M$ at $p$. The Gauss map can always be defined locally. It can be defined globally if the hypersurface is orientable.  

The differential $d\nu$ of the Gauss map $\nu$ can be used to define a type of extrinsic curvature, known as the {\it shape operator} or Weingarten map.  Since at each point $p\in M$, the tangent space $T_{p}M$ is an inner product space, the shape operator $S_{p}$ can be defined as a linear operator on this space by the formula:
\begin{align} g(S_{p}v,w)=g(d\nu(v),w),\;\; v,w\in T_{p}M,\end{align}
 where $g$ denotes the metric tensor on $M$ induced from the Euclidean metric on $\mathbb E^{n+1}$. 
The second fundamental form $\sigma$ is related with the shape operator $S$ by 
\begin{align} g(\sigma(v,w),\xi(p))=g (S_{p}(v),w)\end{align}
for tangent vectors $v,w$ of $M$ at $p$.
The eigenvalues of the shape operator $S_{p}$ are called the principal curvatures. 

The determinant of the shape operator  $S_{p}$ is called the {\it Gauss-Kronecker curvature}, which is denoted by $G(p)$. Thus the Gauss-Kronecker curvature $G(p)$ is nothing but the product of the principal curvature at $p$.
When  $n=2$, the Gauss-Kronecker curvature is simply called the {\it Gauss curvature}, which is intrinsic due to well-known Gauss' theorema egregium. 

For an $n$-input production function $f$, we put 
\begin{align} W=\sqrt{1+\text{\small$\sum$}_{i=1}^{n} f_{x_i}^{2}}.\end{align}

We recall the following well-known lemma  (see, e.g.  \cite{c}).

\begin{lemma} \label{L:3.1} The Gauss-Kronecker curvature $G$ of the graph of an $n$-input production function $f({\bf x})$ is given by 
\begin{align}\label{3.5}G=\frac{\det (f_{x_ix_j})}{W^{n+2}}.\end{align}
\end{lemma}

Also, we recall the following result from \cite{c2012-2}.
\vskip.1in

\noindent {\bf Hessian Determinant Formula}. {\it The  determinant of the Hessian matrix  of a composite function $f=F (h_{1}(x_{1})+\cdots+ h_{n}(x_{n}))$ is given by}
 \begin{equation}\begin{aligned}&  \det(H(f))=(F')^{n}h_{1}''\cdots h_{n}''
 +(F')^{n-1}F'' \sum_{j=1}^{n} h''_{1}\cdots h''_{j-1} h'_{j}{}^{2} h''_{j+1}\cdots h''_{n}.
\end{aligned}\end{equation}
 
\vskip.1in

The next result provides a very simple geometric characterization of generalized ACMS and Cobb-Douglas production functions with degree of homogeneity one.

\begin{theorem} \label{T:4.1} Let $f({\bf x})=F(h_1(x_1)+\cdots +h_n(x_n))$ be a quasi-sum production function satisfying the  constant elasticity of substitution property. Then the graph of $f$ has vanishing Gauss-Kronecker curvature if and only if $f$ is one of the following functions:

\begin{enumerate}

\item a linearly homogeneous generalized ACMS production function given by
\begin{align}\label{3.4} f({\bf x})=\gamma \Big(\sum_{i=1}^{n}a_{i}^{\rho}x_{i}^{\rho}\Big)^{\! \frac{1}{\rho}}, \;\; \gamma, a_1,\ldots,a_n,\rho\ne 0;\end{align}

\item  a linearly homogeneous generalized Cobb-Douglas production function given by
\begin{align}\label{4.31} f({\bf x})=\gamma x_{1}^{\alpha_{1}}\cdots x_{n}^{\alpha_{n}},\;\;\sum_{i=1}^{n}\alpha_{i}=1, \end{align} 
\end{enumerate}
where $\gamma,\alpha_1,\ldots,\alpha_n$ are nonzero constants.
\end{theorem}
\begin{proof} Let $$f({\bf x})=F(h_1(x_1)+\cdots +h_n(x_n))$$ be a quasi-sum production function satisfying the CES property. Then according to Theorem \ref{T:1.1}, $f$ is a homothetic generalized ACMS production function given by \eqref{1.5},
or a homothetic generalized Cobb-Douglas production function given by \eqref{11.5},  or a two-input production function of the form $f=F\big(\frac{x_2}{x_1}\big)$.

First, let us assume that  $f$ is two-input  production function given by $f=F\big(\frac{x_2}{x_1}\big)$ for some strictly increasing function $F$. If the graph of $f$ has vanishing Gauss-Kronecker curvature, then it follows from \eqref{1.6} and Lemma \ref{L:3.1} that  $x_{1}^{-4} F'\big(\frac{x_2}{x_1}\big)^{2}=0$. Hence we find
$F'=0$ which is a contradiction. Therefore this case is impossible.

Next, let us assume that $n\geq 2$ and $f$ is a homothetic generalized ACMS production function given by \eqref{1.5}. Suppose that the graph of $f$ has  vanishing Gauss-Kronecker curvature, then by applying   \eqref{1.5},  Lemma \ref{L:3.1} and the Hessian Determinant Formula, we conclude that
\begin{align}\label{3.7} F'(u)=(\sigma-1)u F''(u)\end{align}
with 
\begin{align}\label{3.8} u=c_1 x_{1}^{\frac{\sigma-1}{\sigma}}+\cdots+c_n x_{n}^{\frac{\sigma-1}{\sigma}}\end{align}
After solving \eqref{3.7} and after applying a suitable translation on $u$, we get 
\begin{align}\label{3.9} F(u)=\alpha u^{\frac{\sigma}{\sigma-1}}\end{align}
for some nonzero constant $\alpha$. Therefore, by combining \eqref{3.8} and \eqref{3.9}, we obtain
\begin{align}\label{3.10} F(u)=\alpha \left(  c_1 x_{1}^{\frac{\sigma-1}{\sigma}}+\cdots+c_n x_{n}^{\frac{\sigma-1}{\sigma}}\right)^{\frac{\sigma}{\sigma-1}}\end{align}
Consequently, the quasi-sum function $f$ is a generalized ACMS production function with degree of homogeneity one.

Finally,  let us assume that $n\geq 2$ and $f$ is a homothetic generalized Cobb-Douglas  function given by \eqref{11.5}. Suppose that the graph of $f$ has  vanishing Gauss-Kronecker curvature. Then it follows from  \eqref{11.5} and  Lemma \ref{L:3.1} that the function $F$ satisfies the following differential equation:
\begin{align}\label{3.11} &(\alpha-1)F'(u)  +\alpha u F''=0,\end{align}
where $$\alpha=\sum_{{i=1}}^{n}\alpha_{i},\;\; u=x_{1}^{\alpha_{1}}\cdots x_{n}^{\alpha_{n}}.$$ 
After solving  \eqref{3.11} we find
\begin{align}\label{3.12} &F(u)  =\gamma u^{\frac{1}{\alpha}}+b\end{align}
for some constants $\gamma,b$ with $\gamma\ne 0$.  Therefore we obtain
\begin{align}\label{3.13} &f({\bf x})  =\gamma (x_{1}^{\alpha_{1}}\cdots x_{n}^{\alpha_{n}})^{\frac{1}{\alpha}}+b\end{align}
 Consequently, up to constants, $f$ is a linearly homogeneous generalized Cobb-Douglas production function.
 
The converse can be verified by direct computation.
\end{proof}

\begin{theorem} \label{T:4.2}Let $f({\bf x})=F(h_1(x_1)+\cdots +h_n(x_n))$ be a quasi-sum production function satisfying the  constant elasticity of substitution property. Then the graph of $f$ is a flat space if and only if either

\begin{enumerate}
\item $f$ is a linearly homogeneous generalized ACMS production function given by
\begin{align}\label{4.40} f({\bf x})=\gamma \Big(\sum_{i=1}^{n}a_{i}^{\rho}x_{i}^{\rho}\Big)^{\! \frac{1}{\rho}},\end{align}
 
\item or $f$ is a linearly homogeneous generalized Cobb-Douglas production function given by
\begin{align}\label{4.41} f({\bf x})=\gamma x_{1}^{\alpha_{1}}\cdots x_{n}^{\alpha_{n}},\;\;\sum_{i=1}^{n}\alpha_{i}=1, \end{align} 
\end{enumerate}
for some nonzero constants $a_{1},\ldots,a_{n},\gamma,\alpha_1,\ldots,\alpha_n$.
\end{theorem}
\begin{proof}  It  is straightforward to verify that the graphs of a linearly homogeneous  generalized ACMS  production function given by \eqref{4.40} and the graph of a linearly homogeneous  generalized Cobb-Douglas production function given by \eqref{4.41} are flat spaces. Therefore, after combining this fact with Theorem \ref{T:4.1} we obtain Theorem \ref{T:4.2}.
\end{proof}

\end{document}